\documentclass[twoside,leqno,10pt, A4]{amsart}
\usepackage{amsfonts}
\usepackage{amsmath}
\usepackage{amscd}
\usepackage{amssymb}
\usepackage{amsthm}
\usepackage{amsrefs}
\usepackage{latexsym}
\usepackage{mathrsfs}
\usepackage{bbm}
\usepackage{amscd}
\usepackage{amssymb}
\usepackage{amsthm}
\usepackage{amsrefs}
\usepackage{latexsym}
\usepackage{mathrsfs}
\usepackage{bbm}
\usepackage{enumerate}
\usepackage{graphicx}
\usepackage{color}
\begin{document}

\newtheorem{theorem}[subsection]{Theorem}
\newtheorem{proposition}[subsection]{Proposition}
\newtheorem{lemma}[subsection]{Lemma}
\newtheorem{corollary}[subsection]{Corollary}
\newtheorem{conjecture}[subsection]{Conjecture}
\newtheorem{prop}[subsection]{Proposition}
\numberwithin{equation}{section}
\newcommand{\mr}{\ensuremath{\mathbb R}}
\newcommand{\mc}{\ensuremath{\mathbb C}}
\newcommand{\dif}{\mathrm{d}}
\newcommand{\intz}{\mathbb{Z}}
\newcommand{\ratq}{\mathbb{Q}}
\newcommand{\natn}{\mathbb{N}}
\newcommand{\comc}{\mathbb{C}}
\newcommand{\rear}{\mathbb{R}}
\newcommand{\prip}{\mathbb{P}}
\newcommand{\uph}{\mathbb{H}}
\newcommand{\fief}{\mathbb{F}}
\newcommand{\majorarc}{\mathfrak{M}}
\newcommand{\minorarc}{\mathfrak{m}}
\newcommand{\sings}{\mathfrak{S}}
\newcommand{\fA}{\ensuremath{\mathfrak A}}
\newcommand{\mn}{\ensuremath{\mathbb N}}
\newcommand{\mq}{\ensuremath{\mathbb Q}}
\newcommand{\half}{\tfrac{1}{2}}
\newcommand{\f}{f\times \chi}
\newcommand{\summ}{\mathop{{\sum}^{\star}}}
\newcommand{\chiq}{\chi \bmod q}
\newcommand{\chidb}{\chi \bmod db}
\newcommand{\chid}{\chi \bmod d}
\newcommand{\sym}{\text{sym}^2}
\newcommand{\hhalf}{\tfrac{1}{2}}
\newcommand{\sumstar}{\sideset{}{^*}\sum}
\newcommand{\sumprime}{\sideset{}{'}\sum}
\newcommand{\sumprimeprime}{\sideset{}{''}\sum}
\newcommand{\sumflat}{\sideset{}{^\flat}\sum}
\newcommand{\shortmod}{\ensuremath{\negthickspace \negthickspace \negthickspace \pmod}}
\newcommand{\V}{V\left(\frac{nm}{q^2}\right)}
\newcommand{\sumi}{\mathop{{\sum}^{\dagger}}}
\newcommand{\mz}{\ensuremath{\mathbb Z}}
\newcommand{\leg}[2]{\left(\frac{#1}{#2}\right)}
\newcommand{\muK}{\mu_{\omega}}
\newcommand{\thalf}{\tfrac12}
\newcommand{\lp}{\left(}
\newcommand{\rp}{\right)}
\newcommand{\Lam}{\Lambda_{[i]}}
\newcommand{\lam}{\lambda}
\newcommand{\af}{\mathfrak{a}}
\newcommand{\sw}{S_{[i]}(X,Y;\Phi,\Psi)}
\newcommand{\lz}{\left(}
\newcommand{\pz}{\right)}
\newcommand{\bfrac}[2]{\lz\frac{#1}{#2}\pz}
\newcommand{\odd}{\mathrm{\ primary}}
\newcommand{\even}{\text{ even}}
\newcommand{\res}{\mathrm{Res}}

\theoremstyle{plain}
\newtheorem{conj}{Conjecture}
\newtheorem{remark}[subsection]{Remark}

\makeatletter
\def\widebreve{\mathpalette\wide@breve}
\def\wide@breve#1#2{\sbox\z@{$#1#2$}%
     \mathop{\vbox{\m@th\ialign{##\crcr
\kern0.08em\brevefill#1{0.8\wd\z@}\crcr\noalign{\nointerlineskip}%
                    $\hss#1#2\hss$\crcr}}}\limits}
\def\brevefill#1#2{$\m@th\sbox\tw@{$#1($}%
  \hss\resizebox{#2}{\wd\tw@}{\rotatebox[origin=c]{90}{\upshape(}}\hss$}
\makeatletter

\title[Subconvexity of a double Dirichlet series over the Gaussian field]{Subconvexity of a double Dirichlet series over the Gaussian field}

\author[P. Gao]{Peng Gao}
\address{School of Mathematical Sciences, Beihang University, Beijing 100191 China}
\email{penggao@buaa.edu.cn}

\author[L. Zhao]{Liangyi Zhao}
\address{School of Mathematics and Statistics, University of New South Wales, Sydney, NSW 2052, Australia}
\email{l.zhao@unsw.edu.au}


\begin{abstract}
 We establish a subconvexity bound for a double Dirichlet series involving with the quadratic Hecke $L$-functions over the Gaussian field.
\end{abstract}

\maketitle

\noindent {\bf Mathematics Subject Classification (2010)}: 11M32, 11M41, 11L40  \newline

\noindent {\bf Keywords}:  Double-Dirichlet series, subconvexity, Hecke $L$-functions, character sums

\section{Introduction}
\label{sec 1}

  Establishing subconvexity bounds for $L$-functions has many important applications in number theory and can also be regarded as an effort towards resolving the well-known Lindel\"of hypothesis.  As multiple Dirichlet series have been shown to be a powerful tool in analytical number theory and have been applied to study various problems including moments of $L$-functions or $L$-functions ratios conjecture (see \cite{DGH, Cech1,D19, DW21}), it is thus desirable to seek for an understanding of some of the finer analytical properties of these multiple Dirichlet $L$-functions such as their subconvexity bounds. \newline

 In \cite{Blomer11}, V. Blomer initiated the study on the subconvexity bounds for a double Dirichlet series given by
\begin{equation}
\label{defDdoubleseries}
  Z_{\mq}(s, w; \psi, \psi') = \zeta^{(2)}(2s+2w-1) \sum_{d \text{ odd}} \frac{L^{(2)}(s, \left(\frac{d}{.}\right)_{\mz} \psi) \psi'(d)}{d^w},
\end{equation}
 where $\psi, \psi'$ are two real Dirichlet characters of conductor dividing $8$, $\left(\frac{d}{.}\right)_{\mz}$ denotes the Kronecker symbol and the superscripts $2$ indicate that the Euler factors at $2$ are being removed. \newline

  It is shown in \cite[Theorem 1]{Blomer11} that for $\Re(s)=\Re(w)=1/2$ and any $\varepsilon>0$, one has
\begin{equation*}
  Z_{\mq}(s, w; \psi, \psi') \ll |sw(s+w)|^{1/6+\varepsilon},
\end{equation*}
  while the convexity bound yields the exponent of $1/4$ instead. \newline

  In \cite{BGL}, V. Blomer, L. Goldmakher and B. Louvel further obtained a uniform subconvexity result for a family of double Dirichlet series involving with $n$-th-order ($n \geq 3$) Hecke $L$-functions of a number field. In \cite{Dahl}, A. Dahl generalized \cite[Theorem 1]{Blomer11} to the case when the corresponding double Dirichlet series is twisted by general Dirichlet characters rather than Dirichlet characters of conductors dividing $8$. \newline

  In this paper, we consider a double Dirichlet series over the Gaussian field $K=\mq(i)$ that is analogue to $ Z_{\mq}(s, w; \psi, \psi')$ given in \eqref{defDdoubleseries}.  To describe such a series, we denote $\mathcal{O}_K=\mz[i]$ for the ring of integers of $K$, $\zeta_K(s)$ the Dedekind zeta function over $K$ and $L(s, \chi_m)$ the Hecke $L$-function associated to the quadratic symbol $\chi_m :=\left(\frac{m}{\cdot} \right)$, to be defined in Section \ref{sec2.4}. It is also known (see \cite[\S 9.7, Lemma 7]{I&R}) that every ideal in $\mathcal{O}_K$ co-prime to $2$ has a unique generator congruent to $1$ modulo $(1+i)^3$ which is called primary. Let $\text{CG}$ be the  abelian group given in \eqref{CG} below consisting of certain real Hecke characters of conductors dividing $(1+i)^5$.  We then define for two $\psi,\psi' \in \text{CG}$,
\begin{equation} \label{defZ}
\begin{split}
 Z(s,w;\psi,\psi') & :=\zeta^{(2)}(2s+2w-1)\sum_{m\odd}\sum_{n\odd}\frac{\leg mn \psi(n)\psi'(m)}{N(m)^wN(n)^s} \\
 & =\zeta^{(2)}(2s+2w-1)\sum_{m\odd}\frac{L^{(2)}(s, \chi_m\psi(n))\psi'(m)}{N(m)^w},
 \end{split}
\end{equation}
 where we denote $N(n)$ for the norm of any $n \in K$ and the superscripts $2$ again mean that the Euler factors at $2$ are being removed. \newline

  As already pointed out in \cite{Blomer11} and \cite{BGL}, there is no obvious concept of analytic conductor or convexity bound in the
context of multiple Dirichlet series. However, based on the discussions  given in \cite{Blomer11} and \cite{BGL}, it is reasonable to define the
analytic conductor of $Z(s,w;\psi,\psi')$ to be
\begin{equation}\label{anacon}
  C(u, t) :=  \left|\tfrac{1}{2} + it \right|^2\cdot  \left|\tfrac{1}{2} + i(u+t) \right|^2 \cdot \left|\tfrac{1}{2} + iu \right|^2.
\end{equation}
  One then expects a convexity bound for $Z(s,w;\psi,\psi')$ such that for any $\varepsilon>0$,
\begin{equation*}
  Z(s,w;\psi,\psi') \ll  C(u,t)^{1/4+\varepsilon}, \quad \mbox{with} \quad \Re(s)=\Re(w)=\tfrac 12, \quad \Im(s)=t \quad \Im(w)=u.
\end{equation*}

  Our aim in this paper to establish the following subconvexity bound of $Z(s,w;\psi,\psi')$ for $\Re(s)=\Re(z)=1/2$.
\begin{theorem}
\label{ThmZbound}
	With the notation as above. For $\Re s = \Re w = 1/2$ any $\varepsilon > 0$, we have
\begin{align}
\label{Zbound}
\begin{split}
	  Z(s, w, \psi, \psi') \ll |sw(s+w)|^{1/3+\varepsilon}.
\end{split}
\end{align}
\end{theorem}

Our proof of Theorem \ref{ThmZbound} follows largely the approach in the proof of \cite[Theorem 1]{Blomer11}. In particular, we shall make a crucial use of the analytical property of $Z(s, w, \psi, \psi')$ given in Section \ref{APZ}.  Moreover, the bound on the right-hand side of \eqref{Zbound} can also be written as $C(u,t)^{1/6+\varepsilon}$ and hence can be regarded as the analogue of the classical subconvexity result of H. Weyl \cite{weyl} for our setting.  Moreover, our Theorem~\ref{ThmZbound} can be regarded as complementing the result in \cite{BGL} which deals with double Dirichlet series involving $n$-th order (with $n \geq 3$) Hecke $L$-functions of a number field, while our result here is concerned with double Dirichlet series with quadratic ($n=2$) Hecke $L$-functions in the Gaussian field. \newline

It would be natural to ask if the analogue of Theorem~\ref{Zbound} can be established for a general number field.  This hinges upon the availability of the analogues of the large sieve result in Lemma~\ref{quartls} and Proposition~\ref{ExtensionOfZ(s,w)}, with the latter giving the analytic properties of $Z(s, w, \psi, \psi')$.  To prove these properties, a sufficient number of functional equations is needed.  The establishment of these functional equations often requires the modification of the original series by attaching some correction factors.  The double Dirichlet series of our present focus is the case in which the correction factors are the simplest.  It is reasonable to expect that the correction factors to become more complicated with a general number field. \newline

Another natural question is whether the result here can lead to an improvement of the first moment result in \cite[Theorem 1.1]{G&Zhao2022-7}.  The subconvexity bound in \eqref{Zbound} can certainly widen our choices of smoothing functions used in \cite{G&Zhao2022-7}.  But it will not reduce the size of the error term in the first moment asymptotic formula in \cite[Theorem 1.1]{G&Zhao2022-7}, as the bound for the contour integrals involved in the moment computation is unaffected by the bound for $Z(s, w, \psi, \psi')$.

\section{Preliminaries}
\label{sec 2}

We first include some auxiliary results needed in the paper.

\subsection{Quadratic residue symbol and quadratic Hecke character}
\label{sec2.4}
   Recall that $K=\mq(i)$ and it is well-known that $K$ has class number one. We denote $U_K=\{ \pm 1, \pm i \}$ and $D_{K}=-4$ for the group of units in $\mathcal{O}_K$ and the discriminant of $K$, respectively. \newline

   For $n \in \mathcal{O}_{K}, (n,2)=1$, we denote $\leg {\cdot}{n}$ for the quadratic residue symbol modulo $n$ in $K$. For a prime $\varpi \in \mz[i]$
with $N(\varpi) \neq 2$, the quadratic residue symbol is defined for $a \in
\mathcal{O}_{K}$, $(a, \varpi)=1$ by $\leg{a}{\varpi} \equiv
a^{(N(\varpi)-1)/2} \pmod{\varpi}$, with $\leg{a}{\varpi} \in \{
\pm 1 \}$. When $\varpi | a$, we define
$\leg{a}{\varpi} =0$.  Then the quadratic symbol is extended
to any composite $n$ with $(N(n), 2)=1$ multiplicatively. We further define $\leg {\cdot}{c}=1$ for $c \in U_K$. Recall that we denote
$\chi_c$ for the symbol $\leg {c}{\cdot}$, where we define $\leg {c}{n}=0$ when $1+i|n$. \newline

  We further write $\psi_j =\leg {j}{\cdot}$ for $j \in \{1, i, 1+i, i(1+i) \}$ and denote
\begin{align}
\label{CG}
  \text{CG}=\{ \psi_j|j=1, i, 1+i, i(1+i) \}.
\end{align}
   Notice that $\text{CG}$ becomes an abelian group upon defining the commutative binary operation by $\psi_i \cdot \psi_{i(1+i)}=\psi_{1+i}$, $\psi_{1+i} \cdot \psi_{i(1+i)}=\psi_i$ and $\psi_j \cdot \psi_j=\psi_1$ for any $j$.  As we shall only evaluate the related characters at primary elements in $\mathcal{O}_K$, the definition of such a product is therefore justified. \newline

  Observe that every primary $c$ can be written uniquely as
\begin{align}
\label{cdcomp}
 c=c_1c_2, \quad \text{$c_1, c_2$ primary and $c_1$ square-free}.
\end{align}
   It is shown in  \cite[Section 2.1]{G&Zhao2022-7} that $\chi_c \cdot \psi$ for any primary $c$ and $\psi \in \text{CG}$ is induced by a primitive Hecke character of conductor dividing $(1+i)^5c_1$. We shall henceforth denote $N(\chi_c \cdot \psi)$ for the norm of the conductor of the primitive Hecke character that induces $\chi_c \cdot \psi$.

\subsection{The approximate functional equation}
\label{sect: apprfcneqn}

  We apply \cite[Theorem 5.3]{iwakow} to arrive at an approximate functional equation for the above Hecke $L$-functions.
\begin{lemma}
\label{lem:AFE}
Let $H(s)$ be an entire, even function with rapid decay in the strip $|\Re(s)| \leq 10$ such that $H(0)=1$. Let $\chi=\chi_c \cdot \psi$ with $c$ primary, square-free and $\psi \in \text{CG}$. We have
\begin{align}
\label{fcneqnL}
\begin{split}
 L(\half+it,\chi)=\sum_{0 \neq \mathcal{A} \subset \mathcal{O}_K} & \frac{\chi(\mathcal{A})}{N(\mathcal{A})^{1/2+it}}G_{t} \left(  \frac {N(\mathcal{A})}{N(\chi)^{1/2}} \right) \\
 & + \left(\frac{N(\chi)}{\pi^2} \right)^{-it}
\frac{\Gamma\left(\thalf- it \right)}{\Gamma\left(\thalf + it \right)}\sum_{0 \neq \mathcal{A} \subset \mathcal{O}_K} \frac{\chi(\mathcal{A})}{N(\mathcal{A})^{1/2-it}}G_{-t} \left(  \frac {N(\mathcal{A})}{N(\chi)^{1/2}} \right),
\end{split}
\end{align}
  where for $\xi>0$, we define
\begin{align}
\label{defG}
 G_t(\xi) = \frac{1}{2 \pi i} \int\limits\limits_{(2)}  \frac {\Gamma(\frac{1}{2}+it+s)}{\Gamma(\frac{1}{2}+it)} (\pi\xi)^{-s}\frac{H(s)}{s} \dif s.
\end{align}
\end{lemma}

 Here we note that it follows from \cite[Proposition 5.4]{iwakow} that for any $A>0$,
\begin{align*}
 G_t(\xi) \ll \left( 1+\frac {\xi}{1+|t|} \right)^{-A}.
\end{align*}

\subsection{The large sieve with quadratic symbols}

The following large sieve inequality for quadratic residue symbols is \cite[Theorem 1]{Onodera} and is the analogue in the Gaussian field to the well-known large sieve result of D. R. Heath-Brown \cite[Theorem 1]{DRHB} for quadratic Dirichlet characters.
\begin{lemma} \label{quartls}
Let $M,N$ be positive integers, and let $(a_n)$ be an arbitrary sequence of complex numbers, where $n$ runs over $\mathcal O_K$. Then we have for any $\varepsilon > 0$,
\begin{equation*}
 \sumstar_{\substack{m \in \mathcal O_K \\N(m) \leq M}} \left| \ \sumstar_{\substack{n \in \mathcal O_K \\N(n) \leq N}} a_n \leg{n}{m} \right|^2
 \ll_{\varepsilon} (MN)^{\varepsilon}(M + N) \sum_{N(n) \leq N} |a_n|^2,
\end{equation*}
  where the asterisks indicate that $m$ and $n$ run over square-free elements of $\mathcal O_K$ that are co-prime to $2$.
\end{lemma}

  As a consequence of Lemma \ref{quartls}, for two given sequences of complex numbers $(a_m), (b_n)$ of absolute values at most $1$, we have for any $\varepsilon >0$,
\begin{align}
\label{cor4toquartls}
  \sum_{\substack{m\leq M\\ m \odd}} \sum_{n \leq N} a_mb_n \left(\frac{n}{m}\right) \ll_{\varepsilon} (MN)^{\varepsilon}(MN^{1/2} + M^{1/2}N).
\end{align}

  The above is an analogue to \cite[Corollary 4]{DRHB} and can be derived from Lemma \ref{quartls} in a similar manner as \cite[Corollary 4]{DRHB} from \cite[Theorem 1]{DRHB}.  Additionally, applying \eqref{cor4toquartls} and following the arguments that lead to \cite[(15)]{Blomer11}, we see that if $\tilde{a}_m, \tilde{b}_m \ll N(m)^{-1/2+\varepsilon}$, then
\begin{equation}
\label{HBbilinear}
  \sum_{\substack{m\leq M\\ m \odd}} \sum_{n \leq N} \tilde{a}_m\tilde{b}_n \left(\frac{n}{m}\right) \ll_{\varepsilon} (M+N)^{1/2+\varepsilon}.
\end{equation}

  We also note the following consequence of \cite[Corollary 1.4]{BGL}, establishing an upper bound for the second moment of quadratic Hecke $L$-functions.
\begin{lemma}
\label{lem:2.3}
Suppose that $s$ is a complex number with $\Re(s) \geq \frac{1}{2}$ and that $|s-1|>\varepsilon>0$. Then for any $\psi \in \text{CG}$,
\begin{align}
\label{L4est}
\sumstar_{\substack{(m,2)=1 \\ N(m) \leq X}} |L(s,\chi_{m}\psi)|^2
\ll (X|s|)^{1+\varepsilon}.
\end{align}
\end{lemma}

  We deduce from \eqref{L4est} and Cauchy's inequality that for any $\Re(s) \geq \frac{1}{2}$ and $|s-1|>\varepsilon>0$,
\begin{align}
\label{L1est}
\sumstar_{\substack{(m,2)=1 \\ N(m) \leq X}} |L(s,\chi_{m})|
\ll X^{1+\varepsilon} |s|^{1/2+\varepsilon}.
\end{align}

\subsection{Analytical Property of $Z(s,w;\psi,\psi')$}
\label{APZ}

    We define for $\psi', \psi'' \in \text{CG}$ and $m, n \in \mz$,
\begin{equation*}
 Z(s,w;\psi,m\psi'+n\psi'') :=mZ(s,w;\psi,\psi')+nZ(s,w;\psi,\psi'').
\end{equation*}
  Our next result quotes \cite[Proposition 3.2]{G&Zhao2022-7} for the analytical property of $Z(s,w;\psi,\psi')$.
\begin{proposition}
\label{ExtensionOfZ(s,w)}
	The functions $Z(s,w;\psi,\psi')$ have a meromorphic continuation to the whole $\mc^2$ with a polar line $s+w=3/2$. There is an additional polar line at $s=1$ with residue $\res_{(1,w)}Z(s,w;\psi,\psi')=\pi \zeta_2(2w)/8$ if and only if $\psi=\psi_1$, and an additional polar line $w=1$ with residue $\res_{(s,1)}Z(s,w;\psi,\psi')=\pi \zeta_2(2s)/8$ if and only if $\psi'=\psi_1$. \newline
	
	The functions $(s-1)(w-1)(s+w-3/2)Z(s,w;\psi,\psi')$ are polynomially bounded in vertical strips, in the sense that for any $C_1>0$, there is a constant $C_2>0$ such that $|(s-1)(w-1)(s+w-3/2)Z(s,w;\psi,\psi')| \ll ((1+\Im(s))(1+\Im(w)))^{C_2}$ whenever $|\Im(s)|, |\Im(w)| \leq C_1$. The functions satisfy the following functional equations:
\begin{align}
\label{AswAws}
 Z(s,w;\psi,\psi')=Z(w,s;\psi',\psi).
\end{align}
Moreover, if $\psi \neq \psi_1$, then
\begin{align}
\label{Zfcneqnpsinot1}
  Z(1-s,s+w-\half;\psi,\psi')=\pi^{1-2s}N(\psi)^{(2s-1)/2}\frac{\Gamma(s)}{\Gamma(1-s)}Z(s,w;\psi,\psi').
\end{align}
For $\psi =\psi_1$, we have
\begin{align}
\label{Zfcneqnpsi1}
\begin{split}
Z& (1-s, s+w-\half; \psi_1,\psi') \\
& = \tfrac 12 \pi^{1-2s}\frac{\Gamma(s)}{\Gamma(1-s)}  \Big ( \tfrac 12(1-2^{-(1-s)})(1-2^{-s})^{-1}Z(s,w;\psi_1,\psi'(\psi_1 +\psi_{i})(\psi_1+ \psi_{1+i})) \\
& \hspace*{1cm} +\tfrac 12(1+2^{-(1-s)})(1+2^{-s})^{-1}Z(s,w;\psi_1,\psi'(\psi_1 +\psi_{i})(\psi_1- \psi_{1+i})) +2^{2s-1}Z(s,w;\psi_1, \psi'(\psi_1-\psi_{i}) \Big ).
\end{split}
\end{align}
\end{proposition}

  To facilitate our treatments in the sequel, we rewrite the functional equations \eqref{Zfcneqnpsinot1} and \eqref{Zfcneqnpsi1} by a successive change of variables: $s \rightarrow 1-s$ and then  $w \rightarrow w+s-1/2$ to arrive at
\begin{align*}
\begin{split}
  Z(s,w;\psi,\psi')=& \pi^{2s-1}N(\psi)^{(1-2s)/2} \frac{\Gamma(1-s)}{\Gamma(s)} Z(1-s,s+w-1/2;\psi,\psi'), \quad \psi \neq \psi_1, \\
  Z(s,w;\psi_1,\psi') =& \tfrac 12 \pi^{2s-1} \frac{\Gamma(1-s)}{\Gamma(s)}  \Big ( \tfrac 12(1-2^{-s})(1-2^{-(1-s)})^{-1}Z(1-s,s+w- \tfrac{1}{2};\psi_1,\psi'(\psi_1 +\psi_{i})(\psi_1+ \psi_{1+i})) \\
& \hspace*{1cm} +\tfrac 12(1+2^{-s})(1+2^{-(1-s)})^{-1}Z(1-s,s+w-\tfrac{1}{2};\psi_1,\psi'(\psi_1 +\psi_{i})(\psi_1- \psi_{1+i})) \\
& \hspace*{1cm} +2^{1-2s}Z(1-s,s+w-\tfrac{1}{2};\psi_1, \psi'(\psi_1-\psi_{i}) \Big ).
\end{split}
\end{align*}

  The above allows us to recast the functional equations \eqref{Zfcneqnpsinot1} and \eqref{Zfcneqnpsi1} in a uniform way such that for some absolute constants $\alpha^{(j)}_{\rho, \rho', \psi, \psi' }$,
\begin{align}
\label{pracfunct}
\begin{split}
  Z(s, w; \psi, \psi')&  = \sum_{j = -5}^2 \sum_{\rho, \rho' \in \text{CG}} \frac{\alpha^{(j)}_{\rho, \rho', \psi, \psi'} 2^{js}\pi^{2s-1}}{4^s-4} \frac{\Gamma(1-s)}{\Gamma(s)} Z(1-s,s+w- \tfrac{1}{2}; \rho, \rho').
  \end{split}
\end{align}

   We now use \eqref{AswAws} to write $Z(1-s,s+w-1/2; \rho, \rho')$ as $Z(s+w-1/2,1-s; \rho', \rho)$ and then apply \eqref{pracfunct} to recast the latter as sums over $Z(3/2-s-w, w; \rho, \rho')=Z(w, 3/2-s-w; \rho', \rho)$.  By  \eqref{pracfunct} again, we further express $Z(w, 3/2-s-w; \rho', \rho)$ as sums over $Z(1-w, 1-s; \rho', \rho)=Z(1-s, 1-w; \rho, \rho')$. After the above process, we see that there exist some absolute constants $\alpha^{(j_1,j_2)}_{\rho, \rho', \psi, \psi' }$ such that
\begin{equation}
\label{finalfunct}
\begin{split}
  Z(s, & w; \psi, \psi')\\
  & = \sum_{j_1, j_2 = -10}^4 \sum_{\rho, \rho' \in \text{CG}} \frac{\alpha^{(j_1,j_2)}_{\rho, \rho', \psi, \psi' } 2^{j_1s + j_2w}\pi^{4s+4w-4}}{(4^s-4)(4^{s+w-1/2}-4)(4^w - 4)}  \frac{\Gamma(1-s)}{\Gamma(s)} \frac{\Gamma(\tfrac{3}{2}-s-w)}{\Gamma(s+w-\tfrac{1}{2})}  \frac{\Gamma(1-w)}{\Gamma(w)}   Z(1-s, 1-w; \rho, \rho').
  \end{split}
\end{equation}

    Notice that the above functional equation implies that the expression
\begin{displaymath}
 \Gamma\left(s\right)\Gamma\left( s+w-\tfrac{1}{2} \right)\Gamma\left(w \right) Z(s, w; \psi, \psi')
\end{displaymath}
  is roughly invariant under $(s, w) \mapsto (1-s, 1-w)$ and therefore justifies in part our choice of the analytical conductor $C(u,t)$ in \eqref{anacon}.

\section{Proof of Theorem \ref{ThmZbound}}
\label{sec Poisson}

\subsection{An approximate functional equation for  $Z(1/2+it , 1/2 + iu; \psi, \psi')$ }
\label{sec: first approxfcneqn}

   We fix real numbers $u, t$ and we define
 \begin{equation*}
  U := (\tfrac 12+|u|)^2, \quad  T := (\tfrac 12+|t|)^2, \quad S := (\tfrac 12 + |u+t|)^2, \quad C:=SU,  \quad X := STU.
\end{equation*}

  In this section, we establish an approximate functional equation for  $Z(1/2+it, 1/2+iu; \psi, \psi')$.
\begin{lemma}
\label{lemZapproxfcneqn}
There are absolutely bounded constants $\lambda^{\pm}_{j, \rho, \rho', \psi, \psi'}(u, t)$ and a smooth, rapidly decaying function $V$  such that for any $\varepsilon > 0$  and any $C' \geq C^{1/2+\varepsilon}$,
\begin{align*}
\begin{split}
 Z(\tfrac{1}{2}+it , \tfrac{1}{2} + iu; \psi, \psi') =  \sum_{\rho, \rho'} \sum_{j =-6}^{10} \sum_{\pm}\lambda^{\pm}_{j, \rho, \rho', \psi, \psi'}(u, t) & \sum_{\substack{d, m \odd \\ N(dm^2) \leq C'}} \frac{L^{(2)}(1/2, \chi_d \rho) \rho'(d)}{N(d)^{1/2\pm iu}N(m)^{1\pm 2i(u+t)}} V\left(\frac{N(dm^2)}{2^j\sqrt{C}}\right) \\
  & + O\left((TC)^{1/4+\varepsilon}\min(S, U)^{-1/2}\right).
  \end{split}
\end{align*}
 \end{lemma}
\begin{proof}
  Let $H(s)=e^{s^2}$ so that it is an even, holomorphic function satisfying $H(0) = 1$ and is of rapid decay on vertical strips. We consider the integral
\begin{equation}
\label{afe1}
\begin{split}
  \frac{1}{2\pi i} \int\limits_{(1)}& \frac{(2^{\frac{1}{2}+iu+z}-1)(2^{\frac{1}{2}+i(u+t)+z}-1)}{(2^{\frac{1}{2}+iu}-1)(2^{\frac{1}{2}+i(u+t)}-1)}
  \frac{(4^{\frac{1}{2} + iu + z} - 4)(4^{\frac{1}{2} + i(u+t) + z} - 4)}{(4^{\frac{1}{2}+iu}-4)(4^{\frac{1}{2}+i(u+t)}-4)}  Z(\tfrac{1}{2}+it, \tfrac{1}{2} + iu + z; \psi, \psi') F_{u, t}(z) H(z) \frac{\dif z}{z},
  \end{split}
\end{equation}
   where
\begin{displaymath}
  F_{u, t}(z) = \frac{1}{2}C^{-z/2} \frac{\Gamma(\frac{1}{2} - iu)\Gamma(\frac{1}{2} - i(u+t))}{\Gamma(\frac{1}{2} + iu)\Gamma(\frac{1}{2} + i(u+t))} \frac{\Gamma(\frac{1}{2} + iu+z)\Gamma(\frac{1}{2} + i(u+t)+z)}{\Gamma(\frac{1}{2} - iu-z)\Gamma(\frac{1}{2} - i(u+t)-z)}  +  \frac{1}{2}C^{z/2}.
\end{displaymath}
  Notice that the possible poles at $z=-1/2 - iu$, $z = -1/2-iu-it$ of $F_{u, t}$ and the possible poles at $z=1/2 - iu$, $z = 1/2-iu-it$ of $Z$ are cancelled out by the two fractions appearing in the integrand of \eqref{afe1}.  Also Stirling's formula (see \cite[(5.112)]{iwakow}) asserts that for bounded $\Re(s)$,
\begin{align}
\label{Stirling}
  \Gamma(s) \ll |s|^{\Re(s)-1/2}e^{-\pi|\Im(s)|/2}.
\end{align}
  This implies that
\begin{align}
\label{Stirlingratio}
  \frac {\Gamma(1-s)}{\Gamma (s)} \ll (1+|s|)^{1-2\Re (s)}.
\end{align}
  It follows from this that $F_{u, t}$ is of moderate growth in vertical strips. Thus by the rapid decay of $H(s)$ on vertical strips, we shift the contour of integration in \eqref{afe1} to $\Re z = -1$ to pick up a pole at $z=0$, whose residue contributes (by noticing that $F_{u, t}(0) = 1$)
\begin{equation}
\label{afe2}
    Z( \tfrac{1}{2}+it, \tfrac{1}{2} + iu; \psi, \psi').
\end{equation}
 For the integral on the new line, we apply the functional equation \eqref{finalfunct} and make a change of variables $z \mapsto -z$ to get
\begin{equation}\label{afe3}
\begin{split}
- \frac{1}{2\pi i} \int\limits_{(1)}  & \frac{(2^{\frac{1}{2}+iu-z}-1)(2^{\frac{1}{2}+i(u+t)-z}-1)}{(2^{\frac{1}{2}+iu}-1)(2^{\frac{1}{2}+i(u+t)}-1)}
  \sum_{j_1, j_2 = -10}^4 \sum_{\rho, \rho' \in \text{CG}}  \alpha^{(j_1, j_2)}_{ \rho, \rho', \psi, \psi' }  \frac{2^{j_1(\frac{1}{2}+it) + j_2(\frac{1}{2} + iu - z)}\pi^{2i(u+t)-2z}}{(4^{\frac{1}{2}+it}-4)(4^{\frac{1}{2}+iu}-4)(4^{\frac{1}{2}+i(u+t)}-4)} \\
  & \times \frac{\Gamma(\frac{1}{2} - it)\Gamma(\frac{1}{2} - i(u+t) + z)\Gamma(\frac{1}{2} - iu + z)}{\Gamma(\frac{1}{2} + it )\Gamma(\frac{1}{2} + i(u+t) -z )\Gamma(\frac{1}{2} + iu -z )} Z( \tfrac{1}{2}-it , \tfrac{1}{2} - iu + z; \rho, \rho') F_{u, t}(-z) H(z) \frac{\dif z}{z}.
\end{split}
\end{equation}
 It follows that the expression in \eqref{afe1} equals the sum of \eqref{afe2} and \eqref{afe3}. Further using the relation
\begin{displaymath}
   \frac{\Gamma(\frac{1}{2} - iu + z)\Gamma(\frac{1}{2} - i(u+t) + z)}{\Gamma(\frac{1}{2} + iu -z )\Gamma(\frac{1}{2} + i(u+t) -z )} F_{u, t}(-z)   = \frac{\Gamma(\frac{1}{2} - iu)\Gamma(\frac{1}{2} - i(u+t))}{\Gamma(\frac{1}{2}+iu)\Gamma(\frac{1}{2}+i(u+t))} F_{-u, -t}(z),
\end{displaymath}
  we get that for certain absolutely bounded complex numbers $\mu^{(j)}_{ \rho, \rho', \psi, \psi' }(u, t)$, the expression in \eqref{afe3} equals
\begin{equation}\label{afe4}
  \begin{split}
 & - \frac{1}{2\pi i} \int\limits_{(1)}   \sum_{j = -10}^6 \sum_{\rho, \rho' \in \text{CG}} \mu^{(j)}_{ \rho, \rho', \psi, \psi' }(u, t) 2^{-jz} \pi^{-2z}     Z(1/2-it, 1/2 - iu + z; \rho, \rho') F_{-u, -t}(z) H(z) \frac{\dif z}{z}.
\end{split}
\end{equation}

  We now apply \eqref{defZ} to expand the function $Z$ appearing in \eqref{afe1} and \eqref{afe4} into Dirichlet series to see that
for some absolutely bounded constants $\mu^{\pm, j}_{ \rho, \rho', \psi, \psi'}(u, t) \in \mc$,
\begin{equation*}
\begin{split}
  Z(\tfrac{1}{2}+it, \tfrac{1}{2}+iu; \psi, \psi') & =  \sum_{\rho, \rho'\in \text{CG} } \sum_{j =-10}^6  \sum_{\pm} \mu^{\pm, j}_{ \rho, \rho', \psi, \psi'}(u, t)\sum_{d, m \odd} \frac{L^{(2)}(\tfrac{1}{2}+it, \chi_d \rho) \rho'(d)}{N(d)^{\frac{1}{2} \pm iu}N(m)^{1\pm 2i(u+t)}} V_{\pm u, \pm t }\left(\frac{2^jN(dm^2)}{\sqrt{C}}\right),
    \end{split}
\end{equation*}
   where
 \begin{equation}
\label{defV}
\begin{split}
  V_{u, t}(\xi) = \frac{1}{2 \pi i} \int\limits_{(1)} C^{-z/2} F_{ u, t}(z)\pi^{-2z}  H(z) \xi^{-z} \frac{\dif z}{z}.
\end{split}
\end{equation}

 Note that it follows from \eqref{Stirlingratio} that we have uniformly for $u,t$,
\begin{equation}
\label{Ha1}
  C^{-z/2} F_{u, t}(z)  \ll (1+|z|)^{4 \Re z}, \quad \Re z \geq 2.
\end{equation}

 The above and the rapid decay of $H$ on vertical strips and \eqref{Ha1} allows us to shift the contour of integration in \eqref{defV} to $\Re(z)=A$ for any $A>2$.  We deduce from this that uniformly in $u$ and $t$,
\begin{displaymath}
  V_{u, t}(\xi) \ll \xi^{-A}.
\end{displaymath}
The above estimation together with \eqref{L1est} implies that for any $\varepsilon > 0$,
\begin{equation}
\label{linearcombi1}
\begin{split}
  Z(\tfrac{1}{2}+it, \tfrac{1}{2}+iu; \psi, \psi')  =  & \sum_{\rho, \rho'\in \text{CG}} \sum_{j =-10}^6 \sum_{\pm} \mu^{\pm, j}_{\rho, \rho', \psi, \psi'}(u, t)\\
  & \hspace*{0.5cm} \times \sum_{\substack{d, m \odd \\ N(dm^2) \leq C^{1/2+\varepsilon}}} \frac{L^{(2)}( \tfrac{1}{2}+it, \chi_d \rho) \rho'(d)}{N(d)^{\frac{1}{2} \pm iu}N(m)^{1\pm 2i(u+t)}} V_{\pm u, \pm t }\left(\frac{2^jN(dm^2)}{ \sqrt{C}}\right) +O\left(T^{1/4+\varepsilon}C^{-A}\right).
  \end{split}
\end{equation}

  Now we want to remove the dependence on $u$ and $t$ of $V_{u, t}$.  To this end, we set
\begin{equation*}
   V (\xi) = \frac{1}{2 \pi i} \int\limits_{(1)}  \pi^{-2z} H(z) \xi^{-z} \frac{\dif z}{z}.
\end{equation*}
 We consider for $\xi \ll C^{\varepsilon}$ the difference
\begin{equation}
\label{diff}
\begin{split}
 & V_{u, t}(\xi) - V(\xi)= \frac{1}{2 \pi i} \int\limits_{(1)}  \left( C^{-z/2} F_{u, t}(z)- 1\right )
  \pi^{-2z} H(z) \xi^{-z} \frac{\dif z}{z}.
  \end{split}
\end{equation}
  Note that the integrand is now holomorphic at $z=0$ and that $V$ is a smooth function that decays rapidly on vertical strips.  Thus we may shift the contour of integration in \eqref{diff} to $\Re(z)=0$. The contribution of the portion of the integral with $|\Im z| > C^{\varepsilon}$ on the new line is $O(C^{-A})$ by \eqref{Ha1} and the rapid decay of $H$ on vertical strips.  For the complement, we note that by a variant of  \cite[Lemma 4.1]{Ha} we have uniformly in $u$ and $t$ that for any $0 < \varepsilon < 1/2$,
\begin{equation} \label{Ha2}
  C^{-iy/2} F_{u, t}(iy) - 1 \ll |y| C^{\varepsilon}\min(S, U)^{-1/2}, \quad y \in \mr,  \ \ |y| < \min \left( \frac {\min(S, U)^{1/2}}{2}, C^{\varepsilon} \right).
\end{equation}

  If $\tfrac{1}{2} \min(S, U)^{1/2} \geq C^{\varepsilon}$, then we may use the above estimation for the remaining integral to see that it is bounded by $O(C^{\varepsilon}\min(S, U)^{-1/2})$. Otherwise, we divide the remaining integral into one over $\Im (z) \leq \tfrac{1}{2} \min(S, U)^{1/2}$ and another $\tfrac{1}{2} \min(S, U)^{1/2} < \Im (z) < C^{\varepsilon}$.  The bound in \eqref{Ha2} gives that the former is bounded by $O(C^{\varepsilon}\min(S, U)^{-1/2})$.  For the latter, the trivial bound $ C^{-iy/2} F_{u, t}(iy) - 1 \ll 1$ and the rapid decay of $H(z)$ renders that it is bounded by $O(\min(S, U)^{-1/2})$. \newline

  We then conclude from the above discussions and \eqref{L1est} that we may replace the weight function $V_{u, t}$ in \eqref{linearcombi1} by $V$ with an error
\begin{displaymath}
  C^{\varepsilon}\min(S, U)^{-1/2} \sum_{N(dm^2) \leq C^{\frac{1}{2} + \varepsilon}} \frac{|L^{(2)}( \tfrac{1}{2}+it, \chi_d \rho)|}{N(dm^2)^{1/2}}  \ll  (TC)^{1/4+\varepsilon}\min(S, U)^{-1/2}.
\end{displaymath}
  This establishes the lemma with $C' = C^{1/2 +\varepsilon}$. The rapid decay of $V$ on vertical strips then implies that the lemma holds for any larger $C'$. This completes the proof.
\end{proof}

\subsection{Estimations for  $D_{\psi, \psi'}(t, u, P; W)$ }
 We deduce from Lemma \ref{lemZapproxfcneqn} that in order to estimate $Z(1/2+it, 1/2+iu; \psi, \psi')$, it suffices to bound
\begin{displaymath}
\sum_{\substack{d, m \odd \\ N(dm^2) \leq C'}} \frac{L^{(2)}(\tfrac{1}{2}+it, \chi_d \rho) \rho'(d)}{N(d)^{1/2\pm iu}N(m)^{1\pm 2i(u+t)}} V\left(\frac{2^jN(dm^2)}{\sqrt{C}}\right).
\end{displaymath}
 We further apply a smooth partition of unity to see that it remains to bound
\begin{displaymath}
  D_{\psi, \psi'}(t, u, P; W) := \sum_{d, m \odd} \frac{L^{(2)}(\tfrac{1}{2} + it, \chi_d\psi) \psi'(d)}{N(d)^{1/2 + iu}N(m)^{1+2i(u+t)}} W\left(\frac{N(dm^2)}{P}\right),
\end{displaymath}
 where $W$ is a smooth function supported on $[1, 2]$ and
\begin{equation}\label{defP}
  1\leq P \leq (US)^{1/2+\varepsilon}.
\end{equation}
  In what follows, we shall always assume that $P$ satisfies \eqref{defP}. Our next result gives a majorant of $D_{\psi, \psi'}(t, u, P; W)$.
\begin{lemma}
\label{Dest1}
With the notation as above and let $T' \geq TX^{\varepsilon}$. We have for any $A>2$,
\[ D_{\psi,  \psi'} (t, u, P; W) \ll \mathcal{R} + P^{-A} , \]
where
\[ \mathcal{R} = \sum_{\pm} \sum_{N(m) \leq P^{1/2+\varepsilon}} \frac{(XN(m))^{\varepsilon}}{N(m)}  \int\limits_{\varepsilon-iX^{\varepsilon}}^{\varepsilon+iX^{\varepsilon}} \int\limits_{\varepsilon-iX^{\varepsilon}}^{\varepsilon+iX^{\varepsilon}}  \Biggl|  \sum_{\substack{N(d_0) \leq N(m)^{-2}P^{1+\varepsilon}\\  d_0 \odd}}\sum_{N(n) \leq (TP)^{1/2+\varepsilon}} \frac{(\chi_{d_0}\psi)(n) \psi'(d_0) N(d_0\psi)^{s/2}}{N(n)^{1/2 \pm it-s}N(d_0)^{1/2 + iu-w}}    \Biggr| \, |\dif w|\, |\dif s| . \]
\end{lemma}
\begin{proof}
 We apply the decomposition \eqref{cdcomp} to write
$d=d_0d^2_1$ with $d_0, d_1$ primary and $d_0$ square-free, then we note that
\begin{equation*}
  L^{(2)}(\tfrac{1}{2} + it, \chi_d\psi) = \prod_{\varpi \mid 2d_1} \left(1-\frac{(\chi_{d_0}\psi)(\varpi)}{N(\varpi)^{1/2+it}}\right) L( \tfrac{1}{2} + it, \chi_{d_0}\psi).
\end{equation*}

  It follows from the above and \eqref{fcneqnL} that we have, for the function $G_{t}$ given by \eqref{defG},
\begin{equation*}
\begin{split}
& | D_{\psi, \psi'}(t, u, P; W)| \\
& \ll \sum_{\pm} \sum_{m, d_1 \odd}\frac{N(d_1)^{\varepsilon}}{N(md_1)}\Biggl| \sum_{  d_0 \odd}\sum_{n} \frac{(\chi_{d_0}\psi)(n) \psi'(d_0)}{N(n)^{1/2 \pm it}N(d_0)^{1/2 + iu}} G_{\pm t}\left(\frac{N(n)}{\sqrt{N(d_0\psi)}}\right) W\left(\frac{N(d_0d_1^2m^2)}{P}\right)\Biggr|.
\end{split}
\end{equation*}

  We can then truncate the sums at  $N(d_0d_1^2m^2) \leq P^{1+\varepsilon}$ and $N(n) \leq (T'P)^{1/2+\varepsilon}$ at the cost of an error $O(P^{-A})$ by the rapid decay of $W$ and $G_t$ on vertical strips.  Notice that the Mellin transform $\widehat{W}$ of $W$ is also an entire function with rapid decay on vertical strips. This allows us to recast $W$ and $G_{t}$ by Mellin inversion to obtain that
\[ | D_{\psi, \psi'}(t, u, P; W)| \ll P^{-A} + \sum_{\pm} \sum_{N(m) \leq P^{1/2+\varepsilon}} \frac{1}{N(m)^{1-\varepsilon}} \int\limits_{(\varepsilon)} \int\limits_{(\varepsilon)} |\mathcal{I} | |\dif w|\, \left|\frac{\dif s}{s}\right| , \]
 where
\[ \mathcal{I} =  \frac {\Gamma(\frac{1}{2}\pm it+s)}{\Gamma(\frac{1}{2}\pm it)} \pi^{-s}H(s) \sum_{\substack{N(d_0) \leq N(m)^{-2}P^{1+\varepsilon}\\  d_0 \odd}}\sum_{N(n) \leq (T'P)^{1/2+\varepsilon}} \frac{(\chi_{d_0}\psi)(n) \psi'(d_0) N(d_0\psi)^{s/2}}{N(n)^{1/2 \pm it+s}N(d_0)^{1/2 + iu+w}}   \widehat{W}(w) \left(\frac{P}{N(m)^2}\right)^w . \]
   We further apply \eqref{Stirling} and the rapid decay of   $\widehat{W}$ to truncate the $s, w$-integration to arrive at the conclusion of the lemma.
\end{proof}

 In what follows we establish another estimation for $D_{\psi, \psi'}(t, u, P;W)$. We apply Mellin inversion to see that
\begin{equation}\label{hence}
  D_{\psi, \psi'}(t, u, P; W) = \frac{1}{2\pi i} \int\limits_{(1)} Z(\tfrac{1}{2}+it , \tfrac{1}{2} + iu + w; \psi, \psi') \widehat{W}(w) P^w \dif w.
\end{equation}
  We seek to express the function $Z$ appearing in the integrand above as a sum of certain integrals. For this, consider for $\Re z \geq 3/2$, $R>0$,
\begin{equation} \label{expression}
- Z( \tfrac{1}{2}+it, z; \psi, \psi') + \frac{1}{2\pi i} \int\limits_{(3)} \frac{4^{1/2+it+ s}-4}{4^{1/2+it}-4} Z(\tfrac{1}{2} +it+ s, z; \psi, \psi') R^s H(s) \frac{\dif s}{s},
\end{equation}
  where $H(s)=e^{s^2}$ is the same function defined in Section \ref{sec: first approxfcneqn}. \newline

  We shift the line of integration in \eqref{expression} to $\Re s = -3$ to note that the possible pole of $Z$ at $s=1/2-it$ is cancelled out by the fraction in the integrand and the pole at $s=1-it- z$  contributes at most  $O(|\zeta_K(1+2s+2it)|R^{1-\Re z}(1+|z+it|)^{-A})$ by the rapid decay of $H$ and Proposition \ref{ExtensionOfZ(s,w)}. We then make a change of variable $s \mapsto -s$ and apply the functional equation \eqref{pracfunct} to the integration on the new line to see that for some absolutely bounded constants $\tilde{\alpha}^{(j)}_{\rho, \rho', \psi, \psi'}(t)$, the integration in \eqref{expression} equals
\begin{equation}\label{expression2}
\begin{split}
  -\frac{1}{2\pi i}  \int\limits_{(3)} \sum_{j = -5}^2  \sum_{ \rho, \rho'  \in \text{CG}} \tilde{\alpha}^{(j)}_{ \rho, \rho', \psi, \psi'}(t) \frac{\Gamma(\frac{1}{2}-it+s)}{\Gamma(\frac{1}{2}+it-s)} & Z(\tfrac{1}{2} -it+ s, z+it-s; \rho, \rho') \frac{H(s)}{2^{js} \pi^{2\pi} R^s} \frac{\dif s}{s} \\
  &+ O(R^{1-\Re z}(1+|it+z|)^{-A}).
  \end{split}
\end{equation}

 We substitute \eqref{expression} and \eqref{expression2} with $R = (PT)^{1/2}$   and $z = 1/2 + iu + w$ into \eqref{hence} to obtain that
\begin{equation}
\label{mainquant}
  D_{\psi, \psi'}(t, u, P; W) = \Delta + \tilde{\Delta} + O(P^{3/4}S^{-A}),
\end{equation}
 where the error term above follows from the rapid decay of $\widehat{W}$ and
\begin{align}
\label{defD}
\begin{split}
  \Delta =&  \left(\frac{1}{2 \pi i}\right)^2 \int\limits_{(1)}\int\limits_{(3)}  \frac{4^{\frac{1}{2}+it+s}-4}{4^{\frac{1}{2}+it}-4} Z(\tfrac{1}{2}+it+s, \tfrac{1}{2}+iu+w; \psi, \psi')\widehat{W}(w)\frac{H(s)}{s} T^{s/2}P^{w+s/2} \dif s\, \dif w,  \\
 \tilde{\Delta}  =&   \sum_{j=-5}^2 \sum_{ \rho, \rho'  \in \text{CG}} \tilde{\alpha}^{(j)}_{ \rho, \rho', \psi, \psi'}(t)  \left(\frac{1}{2 \pi i}\right)^2 \int\limits_{(1)}\int\limits_{(3)}\frac{\Gamma(\frac{1}{2}-it+s)}{\Gamma(\frac{1}{2}+it-s)} \\
& \hspace*{3cm} \times Z(\tfrac{1}{2}-it +s, \tfrac{1}{2}+i(u+t)+w-s; \rho, \rho')\widehat{W}(w)   \frac{H(s)}{2^{js} \pi^{2s} s} T^{-s/2} P^{w-s/2} \dif s\, \dif w.
\end{split}
\end{align}

  We now shift the $w$-integration in the above expression for $\Delta$ to $\Re w = -1$ to encounter a pole at $w=1/2-iu$, whose residue contributes $O(P^{1/2+\varepsilon}T^{\varepsilon}U^{-A})$, after applying Proposition \ref{ExtensionOfZ(s,w)}. By the rapid decay of $H(s)$ and $\widehat{W}$, we may shift the remaining integral over $s$ to $\Re(s)=\varepsilon>0$. Then a change of variable $w \mapsto -w$ for the integral on the new line gives that
\begin{equation}\label{defDa}
\begin{split}
\Delta =  \left(\frac{1}{2 \pi i}\right)^2 \int\limits_{(1)}\int\limits_{(\varepsilon)}  \frac{4^{\frac{1}{2}+it+s}-4}{4^{\frac{1}{2}+it}-4} Z(\tfrac{1}{2}+it+s, \tfrac{1}{2}+iu-w; \psi, \psi') & \widehat{W}(-w)\frac{H(s)}{s}T^{\frac{s}{2}}P^{-w+\frac{s}{2}} \dif s\, \dif w \\
& + O(P^{1/2+\varepsilon}T^{\varepsilon}U^{-A}).
\end{split}
\end{equation}

Now we apply the functional equations \eqref{AswAws} and \eqref{pracfunct} successively to see that for certain absolutely bounded constants $\beta^{(j)}_{\rho, \rho', \psi, \psi'}(u)$,
\begin{equation}\label{newfe1a}
\begin{split}
Z &(\tfrac{1}{2} + it+ s, \tfrac{1}{2} + iu - w; \psi, \psi') \\
& = \sum_{j = -5}^2 \sum_{\rho, \rho' \in \text{CG}}  \beta^{(j)}_{\rho, \rho', \psi, \psi' }(u)\frac{2^{-j w}\pi^{-2w}}{4^{\frac{1}{2}+iu-w} - 4} \frac{\Gamma(\frac{1}{2}-iu+w)}{\Gamma(\frac{1}{2}+iu-w)}   Z(\tfrac{1}{2}-iu+w, \tfrac{1}{2}+i(u+t)-w+s; \rho, \rho').
\end{split}
\end{equation}

  Similarly, we apply the functional equations \eqref{AswAws}, \eqref{pracfunct} and then \eqref{AswAws} successively to see that
for certain absolutely bounded constants $\tilde{\beta}^{(j)}_{ \tilde{\rho}, \tilde{\rho}', \rho, \rho'}(u, t)$,
\begin{equation}\label{newfe2a}
\begin{split}
&Z(\tfrac{1}{2} -it + s, \tfrac{1}{2} + i(u+t) + w-s; \rho, \rho') \\
& = \sum_{j = -5}^2 \sum_{\tilde{\rho}, \tilde{\rho}' \in \text{CG}}  \tilde{\beta}^{( j)}_{ \tilde{\rho}, \tilde{\rho}', \rho, \rho' }(u, t)\frac{2^{j(w-s)}\pi^{2(w-s)}}{4^{\frac{1}{2}+i(u+t)+w-s} - 4} \frac{\Gamma(\frac{1}{2}-i(u+t)-w+s)}{\Gamma(\frac{1}{2}+i(u+t)+w-s)} Z(\tfrac{1}{2}+iu+w, \tfrac{1}{2}-i(u+t)-w+s; \tilde{\rho}, \tilde{\rho}').
\end{split}
\end{equation}
 We substitute \eqref{newfe1a} into \eqref{defDa} and apply \eqref{defZ} to write $Z$ into double Dirichlet series to see that
\begin{equation}\label{finalD}
\begin{split}
  \Delta =  \sum_{\rho, \rho' \in \text{CG}}  \sum_{n, d \odd} & \frac{\chi_d(n)\rho(n)\rho'(d)}{N(n)^{1/2-iu} N(d)^{1/2+i(u+t)}} \,\,   \left(\frac{1}{2 \pi i}\right)^2 \int\limits_{(1)}\int\limits_{(\varepsilon)}\zeta^{(2)}_{K}(2s+2it+ 1) \\
  & \times \sum_{j = -5}^2 \beta^{(j)}_{ \rho, \rho', \psi, \psi' }(u)  \frac{4^{\frac{1}{2}+it+s}-4}{4^{\frac{1}{2}+it}-4}   \frac{2^{-jw}\pi^{-2w}}{4^{\frac{1}{2}+iu-w} - 4} \frac{\Gamma(\frac{1}{2}-iu+w)}{\Gamma(\frac{1}{2}+iu-w)}    \\
&  \times \left(\frac{N(n)P}{N(d)}\right)^{-w} \Biggl(\frac{N(d)}{\sqrt{T P}}\Biggr)^{-s} \widehat{W}(-w) \frac{H(s)}{s}  \dif s\, \dif w + O(P^{1/2+\varepsilon}T^{\varepsilon}U^{-A}).
  \end{split}
\end{equation}

  We also substitute \eqref{newfe2a} into $\tilde{\Delta}$ given in \eqref{defD} and apply \eqref{defZ} to write $Z$ into double Dirichlet series to see that for certain absolutely bounded constants $ \gamma^{(j_1, j_2)}_{ \rho, \rho', \psi, \psi' }(u, t)$,
\begin{equation}\label{finalDD}
\begin{split}
 &\tilde{\Delta}  = \sum_{\rho, \rho' \in \text{CG}}  \sum_{n, d \odd} \frac{\chi_d(n)\rho(n)\rho'(d)}{N(n)^{1/2+iu} N(d)^{1/2-i(u+t)}} \,\,   \left(\frac{1}{2 \pi i}\right)^2 \int\limits_{(1)}\int\limits_{(3)} \sum_{j_1, j_2 = -5}^2 \gamma^{(j_1, j_2)}_{ \rho, \rho', \psi, \psi' } (u, t)\zeta^{(2)}_{K}(2s-2it+ 1) \\
 &\times \frac{2^{-j_1s + j_2(  w-s)}\pi^{2w-4s}}{4^{\frac{1}{2}+i(u+t)+w-s} - 4}\frac{\Gamma(\frac{1}{2}-it+s)}{\Gamma(\frac{1}{2}+it-s)}  \frac{\Gamma(\frac{1}{2}-i(u+t)-w+s)}{\Gamma(\frac{1}{2}+i(u+t)+w-s)} \widehat{W}(w)       \frac{H(s)}{s} \left(\frac{N(n)}{N(d)P}\right)^{-w} \left(N(d)\sqrt{T P}\right)^{-s} \dif s\, \dif w.
\end{split}
\end{equation}

 We now substitute \eqref{finalD} and \eqref{finalDD} into \eqref{mainquant} to obtain an approximate functional equation for $D_{\psi, \psi'}(t, u, P; W)$. To further simplify the resulting expression, we write for $\varepsilon>0$,
\begin{equation*}
  T' \geq TX^{\varepsilon}, \quad S' \geq SX^{\varepsilon}, \quad U' \geq UX^{\varepsilon}, \quad X'= S'T'U'.
\end{equation*}

  We note that upon shifting the $s$-integration in \eqref{finalD} to $\Re s = A$ yields that the double integration there is
  \[ \ll \frac{N(d)U}{N(n)P} \left(\frac{\sqrt{TP}}{N(d)}\right)^A. \]
 It follows that restricting the $d$-sum to the range with $N(d) \leq (T' P)^{1/2}$ incurs an error of size $O(X^{-A})$.  We next shift the $w$-integration in \eqref{finalD} to $\Re w = A$ to see that the double integral is $\ll (U\sqrt{T}/(\sqrt{P} N(n))^A$ on $\Re s = \Re w = A$, upon using \eqref{Stirlingratio}.  Thus, at the cost of an error term of size $O(X^{-A})$, we may restrict the $n$-sum to
\begin{displaymath}
  N(n) \leq  \frac{(T')^{1/2}U'}{P^{1/2}}.
\end{displaymath}
Along the same line, we shift the $s$-contour in \eqref{finalDD} to $\Re s = A$ and apply \eqref{Stirlingratio} to see that it suffices to keep only the terms in the $d$-sum with
\begin{displaymath}
  N(d) \leq \frac{(T')^{1/2}S'}{P^{1/2}}
\end{displaymath}
and the error in this abridgement is $O(X^{-A})$.  Now shifting the $w$-contour in \eqref{finalDD} to $\Re w = A$ and applying \eqref{Stirlingratio} lead to that, with an error term of size $O(X^{-A})$, we may also shorten the $n$-sum to the range
\begin{displaymath}
  N(n) \leq   (T'P)^{1/2}.
\end{displaymath}

  After truncating the double sums in \eqref{finalD} and \eqref{finalDD} as above, we then shift the contours back to $\Re s = \Re w = \varepsilon$ and interchange the now finite $d, n$-double sum with the absolutely convergent $s, w$-double integral. By the rapid decay of $\widehat{W}$ and $H$ on vertical strips, we may further cut short the $s, w$-integration to arrive at the following estimation of $D_{\psi, \psi'}(t, u, P; W)$.
\begin{lemma}
\label{Dest2}
 With the notation as above. For any $\varepsilon>0, A>2$, we have
\begin{displaymath} \label{Dest2eq}
\begin{split}
  D_{\psi, \psi'}(t, u, P; W)  \ll  (X')^{\varepsilon} & \max_{\rho, \rho' \in \text{CG}}   \int\limits_{\varepsilon-i(X')^{\varepsilon}}^{\varepsilon+i(X')^{\varepsilon}} \int\limits_{\varepsilon-i(X')^{\varepsilon}}^{\varepsilon+i(X')^{\varepsilon}}   \Biggl|\sum_{\substack{n, d \odd \\ N(d)\leq (T'P)^{1/2}  \\ N(n) \leq (T'/P)^{1/2}U' }} \frac{\chi_d(n)\rho(n)\rho'(d)}{N(n)^{1/2-iu+w} N(d)^{1/2+i(u+t)-w+s}}\Biggr|  \\
  & +     \Biggl| \sum_{\substack{n, d \odd \\ d\leq (T'/P)^{1/2}S'  \\ n \leq (T'P)^{1/2}  }} \frac{\chi_d(n)\rho(n)\rho'(d)}{N(n)^{1/2+iu+w} N(d)^{1/2-i(u+t)-w+s}}\Biggr| \, |\dif s\, \dif w|   +O\left( \frac{P^{3/4}}{S^A}+\frac{P^{1/2+\varepsilon}T^{\varepsilon}}{U^A} \right) .
   \end{split}
\end{displaymath}
\end{lemma}

\subsection{Completion of the proof}

In a manner similar to \cite[Section 6]{Blomer11}, we may assume that
\begin{equation*}
  T \leq U \asymp S.
\end{equation*}

  We then apply \eqref{HBbilinear} to estimate the character sums involved in Lemma \ref{Dest1}.  More specifically, we use \eqref{HBbilinear} to bound the integrand in $\mathcal{R}$ which is defined in the statement of Lemma~\ref{Dest1}, absorbing all except $\chi_{d_0}(n)$ into the coefficients $\tilde{a}_m$ or $\tilde{b}_n$ in \eqref{HBbilinear}.   Similarly, the bound in \eqref{HBbilinear} can be applied to the integrand on the right-hand side of \eqref{Dest2eq}.  As a result, we get
\begin{displaymath}
  D_{\psi, \psi'}(t, u, P; W) \ll U^{\varepsilon} \min \left(P^{1/2} + (TP)^{1/4}, (TP)^{1/4} +\left(\frac{T}{P}\right)^{1/4} U^{1/2}\right).
\end{displaymath}
   We deduce from the above and Lemma \ref{lemZapproxfcneqn} that
\begin{displaymath}
  Z(\tfrac{1}{2}+it, \tfrac{1}{2}+iu, \psi, \psi') \ll \frac{(TU)^{1/4+\varepsilon}}{U^{1/2}} + U^{\varepsilon} \max_{P \ll U}\left( (TP)^{1/4} + \min \left(P^{1/2} ,  \left(\frac{T}{P}\right)^{1/4} U^{1/2}\right)\right).
\end{displaymath}
  It is readily seen that the minimum above equals the first term if and only if $P \leq U^{2/3}T^{1/3}$. From this we deduce that
\begin{displaymath}
   Z(\tfrac{1}{2}+it, \tfrac{1}{2}+iu, \psi, \psi') \ll U^{1/3+\varepsilon}T^{1/6}.
\end{displaymath}
  This implies \eqref{Zbound} and completes the proof of Theorem \ref{ThmZbound}.

\vspace*{.5cm}

\noindent{\bf Acknowledgments.}   P. G. is supported in part by NSFC grant 11871082 and L. Z. by the Faculty Silverstar Grant PS65447 at the University of New South Wales (UNSW).  The authors would like to thank the anonymous referee for his/her insightful comments that resulted in some improvements of the presentation of the paper.

\bibliography{biblio}
\bibliographystyle{amsxport}

\end{document}